\documentclass{daj}
\usepackage{amsmath,amssymb,amsthm,mathtools,bbm}
\usepackage[numbers]{natbib}
\usepackage{mathrsfs}  

\def\Cond{\,\vert\,}
\newcommand\one{\mathbbm{1}}
\newcommand\nm[1]{\left\lVert#1\right\rVert}

\theoremstyle{plain}
\newtheorem*{theorem*}{Theorem}
\newtheorem{theorem}{Theorem}[section]
\newtheorem{lemma}[theorem]{Lemma}
\newtheorem{claim}[theorem]{Claim}
\newtheorem{proposition}[theorem]{Proposition}
\newtheorem*{claim*}{Claim}

\def\N{\mathbb{N}}

\def\R{\mathbb{R}}
\def\P{\mathbb{P}}
\def\E{\mathbb{E}}
\def\Ss{\mathbb{S}}
\def\Bb{\mathbb{B}}
\def\C{\mathcal}
\def\B{\mathbf}
\def\Scr{\mathscr}

\let\eps\varepsilon

\dajAUTHORdetails{%
  title = {Balancing Sums of Random Vectors}, 
  author = {Juhan Aru, Bhargav Narayanan, Alex Scott, and Ramarathnam Venkatesan},
  plaintextauthor = {Juhan Aru, Bhargav Narayanan, Alex Scott, Ramarathnam Venkatesan},
}   

\dajEDITORdetails{%
   year={2018},
   number={4},
   received={10 August 2017},   
   published={12 March 2018},  
   doi={10.19086/da.3108},       
}   

\begin{document}

\begin{frontmatter}[classification=text]


\author[ja]{Juhan Aru}
\author[bn]{Bhargav Narayanan}
\author[as]{Alex Scott}
\author[rv]{Ramarathnam Venkatesan}

\begin{abstract}
We study a higher-dimensional `balls-into-bins' problem. An infinite sequence of i.i.d.\ random vectors is revealed to us one vector at a time, and we are required to partition these vectors into a fixed number of bins in such a way as to keep the sums of the vectors in the different bins close together; how close can we keep these sums almost surely? This question, our primary focus in this paper, is closely related to the classical problem of partitioning a sequence of vectors into balanced subsequences, in addition to having applications to some problems in computer science.
\end{abstract}
\end{frontmatter}

\section{Introduction}
In this note, we consider the following partitioning problem. Let $\B{V}(\mu) = (V_n)_{n\ge 1}$ be a sequence of independent random vectors, all distributed according to some common probability distribution $\mu$ on the $d$-dimensional Euclidean unit ball $\Bb^d \subset \R^d$. The elements of this i.i.d.\ sequence $\B{V}(\mu)$ are revealed to us in order, one vector at a time. Each time a new vector is revealed to us, we are required to assign this vector to one of a fixed number of bins $\B{B}_1, \B{B}_2, \dots, \B{B}_k$ before seeing the next vector in the sequence. By adopting a suitable strategy to assign vectors to bins, how `close together' can we keep the sums of the vectors in the different bins? 

A more precise formulation of this question is as follows. For each $1 \le i \le k$, let $B^0_i = 0$, and for a positive integer $n \in \N$, let $B_i^n$ denote the sum of the vectors in the bin $\B{B}_i$ at time $n$; in other words, 
\[B^n_i = \sum_{j = 1}^n V_j \one_{\{V_j \in \B{B}_i\}}.\]
A \emph{partitioning strategy} is a (possibly randomised) map from $(\R^d)^{k+1}$ to the set of bins $\{\B{B}_1, \B{B}_2, \dots, \B{B}_k\}$ which, given the vectors $B^n_1, B^n_2, \dots, B^n_k$ and $V_{n+1}$, tells us which bin  $V_{n+1}$ should be assigned to; in the language of computer science, a partitioning strategy is an \emph{online algorithm} for assigning vectors to bins. We shall mainly be interested in partitioning strategies that minimise, for (large) $T \ge 0$,  the quantity
\[D(T) = \max_{1 \le n \le T} \max_{1 \le i,j \le k}  \,\nm{B_i^n - B_j^n},\]
the largest observed Euclidean distance between a pair of bins up to time $T$. In this paper, we shall mostly be concerned with the asymptotic behaviour of $D(T)$ as $T \to \infty$ while the dimension $d \ge 1$, the number of bins $k \ge 2$ and the distribution $\mu$ remain fixed, so it is perhaps worth emphasising that we choose to work with the Euclidean norm and to track the largest observed distance between any pair of bins purely for concreteness; indeed, all of our results hold as stated, albeit with different implied constants, for any choice of norm, and for any well-defined notion that tracks how close together the various bins are, such as the largest distance between any bin and the average of the bins, for example.

The fact that serves as the starting point for the work here is the classical result (see~\citep{textbook}, for instance) that by assigning vectors to bins uniformly at random, one can always ensure that $D(T) = O(\sqrt{T\log T})$; we shall attempt to quantify by how much one can hope to improve on this. We discuss two other motivations for studying the problem at hand below.

First, the related problem of partitioning a \emph{deterministic} sequence of vectors from the $d$-dimensional unit ball into a fixed number of `balanced subsequences' has a rich history and may be traced back to an old question due to Riemann and L\'evy that was subsequently answered by Steinitz~\citep{Steinitz}; various forms of this problem have since been investigated and we refer the interested reader to the survey of B{\'a}r{\'a}ny~\citep{sum-survey}. We mention one result in this area that is relevant to the problem at hand. Let $\B{V} = (V_n)_{n\ge 1}$ be \emph{any} sequence of vectors lying in the $d$-dimensional unit ball $\Bb^d \subset \R^d$, and consider the problem of assigning each element of this sequence of vectors to one of a fixed number of bins $\B{B}_1, \B{B}_2, \dots, \B{B}_k$ as before, except using a \emph{prescient partitioning strategy}: by a prescient partitioning strategy, we mean a strategy that is allowed to see the entire sequence $\B{V}$ ahead of time. Improving on an earlier result of Doerr and Srivastav~\citep{offline-Kpart1}, B{\'a}r{\'a}ny and Doerr~\citep{offline-Kpart2} proved that there exists a prescient partitioning strategy that ensures that $D(T) \le Cd$ uniformly in $T$ for any $d,k \in \N$ with $k \ge 2$ and any sequence $\B{V}$ as above, where $C \approx 4.001$ is a universal constant (though it should be noted that while these prescient strategies do require `knowledge of the future', they only need to look $\Theta(d)$ vectors ahead). In the light of this fact, it is natural to ask what changes when we are required to partition $\B{V}$ without any knowledge of the future.

Next, the question we study here arises naturally in the context of load-balancing, resource allocation and scheduling problems in large scale computation. In these settings, there is a finite set of servers and a sequence of incoming jobs. Jobs must be allocated to servers as and when they arrive, and each job consumes certain quantities of the different resources (memory or processing power, for example) available on the server it is assigned to. Ideally, one would like to allocate the jobs in a `balanced' fashion where the total load on each server is roughly the same; see the survey of Azar~\citep{lb-survey} for a short introduction to this area. Partitioning strategies that perform well in the `worst-case' can often be suboptimal in practice since the empirical distribution of the incoming jobs is typically random (and not adversarial). Therefore, various probabilistic models for these problems have been studied over the last twenty years and a number of results have been proved in different settings. These results, for the most part, deal with the one-dimensional case of the problem we study and address various questions about distributing (possibly weighted) balls into bins; for a small sample of the existing literature, see~\citep{lb-1, lb-2, lb-3, lb-5}. The higher-dimensional problem we study here is not only inherently interesting, but also exhibits genuinely different behaviour compared to the one-dimensional problem, as we shall shortly see. Finally, let us make two remarks with practical applications in mind: first, in practice, we lose no generality by assuming that the `job-vectors' come from the unit ball as our results remain valid after any suitable (finite) rescaling; second, fine tuning our partitioning strategies for specific distributions can make the strategies fragile, so we focus on results that either hold for all probability distributions on the unit ball, or are robust for a wide class of `nice' distributions.

\section{Our results}\label{s:results}
A few remarks about notation are in order before we state our results. In what follows, we write $[n]$ for the set $\{1, 2, \dots, n\}$. We use $\langle \cdot , \cdot \rangle$ to denote the standard inner product in $\R^d$ and $\nm{\cdot}$ to denote the associated Euclidean norm. Also, we write $\lambda_d$ for the $d$-dimensional Lebesgue measure. We shall make use of standard asymptotic notation; in what follows, the variable tending to infinity will always be $T$ unless we explicitly specify otherwise. Constants suppressed by the asymptotic notation are allowed to depend on the fixed parameters ($d$, $k$ and $\mu$) but not on $T$. Finally, we use the term \emph{with high probability} to mean with probability tending to $1$ as $T \to \infty$.

Our first result, a strategy-agnostic lower bound on $D(T)$, will serve as a useful benchmark.
\begin{proposition}\label{lowbound}
Fix $d,k \ge 2$, and let $\mu$ be the uniform distribution on $\Bb^{d} \subset \R^d$. Regardless of the partitioning strategy used to partition $\B{V}(\mu)$ into $k$ bins, almost surely,
\[\liminf_{T\to\infty} \left(\frac{D(T) (\log\log T)^{1/2}}{(\log T)^{1/2}} \right) = \Omega(1).\]
\end{proposition}

The above proposition immediately highlights the difference between the one-dimensional partitioning problem and the same problem in higher dimensions. Indeed, if we have a sequence of i.i.d.\ vectors distributed according to some common distribution $\mu$ on $\Bb^1 = [-1,1]$, then the trivial partitioning strategy that assigns a vector $V$ to the bin with the largest sum if $V < 0$ and to the bin with the smallest sum if $V>0$ shows that in one dimension, we may uniformly ensure that $D(T) \le 1$ for any number of bins $k$ and any distribution $\mu$.

In an attempt to match the lower bound in Proposition~\ref{lowbound}, we consider two different partitioning strategies below. Note that for any $d, k \in \N$ and any distribution $\mu$ on $\Bb^d$, we may, as discussed earlier, toss each element of $\B{V}(\mu)$ into one of the $k$ bins uniformly at random and thereby ensure with high probability that $D(T) = O(\sqrt{T\log T})$. To improve on this trivial bound, we shall have to work a bit harder.

The first partitioning strategy we propose, which we call the \emph{inner product rule}, is as follows: simply assign $V_{n+1}$ to the bin $\B{B}_i$ for which $\langle V_{n+1}, B_i^n \rangle$ is minimal, breaking ties arbitrarily. Intuitively, this should keep the bins close together since we always add a vector to the bin it is `most opposite' to. We shall show that the inner product rule is a near-optimal partitioning strategy for any reasonably well-behaved probability distribution. Recall that a measure $\mu$ on $\R^d$ is \emph{H\"older continuous (with respect to the Lebesgue measure)} if there exist constants $K, \alpha >0$ such that $\mu(S) \le K\lambda_d(S)^\alpha$ for any measurable set $S \subset \R^d$; the following bounds for the inner product rule are essentially tight.
\begin{theorem}\label{ipupbound}
Fix $d,k \in \N$ with $k \ge 2$, let $\mu$ be a probability distribution on $\Bb^d \subset \R^d$, and suppose that we partition $\B{V}(\mu)$ into $k$ bins using the inner product rule. Then almost surely,
\begin{equation}\limsup_{T\to\infty}\left(\frac{D(T)}{(\log T)^{1/2}} \right) = O(1);\label{gen_mu}
\end{equation}
if $\mu$ is additionally H\"older continuous, then almost surely,
\begin{equation}
\limsup_{T\to\infty}\left(\frac{D(T) (\log\log T)^{1/2}}{(\log T)^{1/2}} \right) = O(1).\label{nice_mu}
\end{equation}
\end{theorem}
In the light of Proposition~\ref{lowbound}, it is immediately clear that~\eqref{nice_mu} is essentially best-possible. The discrepancy between the bound~\eqref{nice_mu} for well-behaved distributions and the bound~\eqref{gen_mu} for arbitrary distributions in Theorem~\ref{ipupbound} is not just an artefact of our proof: somewhat surprisingly, the next proposition demonstrates the existence of (slightly pathological) distributions which show that~\eqref{gen_mu} is also nearly tight.

\begin{proposition}\label{badexist}
For every increasing function $\omega \colon  \R_{>0} \to \R_{>0}$ that grows without bound, there exists a probability distribution $\mu_\omega$ on $\Bb^2 \subset \R^2$ for which the following holds. If we partition $\B{V}(\mu_\omega)$ into two bins using the inner product rule, then almost surely,
\[\limsup_{T\to\infty}\left(\frac{D(T) \omega(T)}{(\log T)^{1/2}} \right) \ge 1.\]
\end{proposition}

The other strategy we consider is motivated by more practical considerations: in applications, where the number of bins $k$ is often large, it is usually too expensive to compute $k$ inner products to make each decision. With this in mind, we investigate the following strategy, a higher dimensional analogue of the `two random choices' strategy studied by Azar, Broder, Karlin and Upfal~\citep{pow2}, which we call the \emph{best-of-two rule}. Unlike the inner product rule, the best-of-two rule is a randomised strategy: given $V_{n+1}$, we choose two bins $\B{B}_i$ and $\B{B}_j$ randomly from the set of all bins (without replacement) and assign $V_{n+1}$ to $\B{B}_i$ if $\langle V_{n+1}, B^n_i \rangle \le \langle V_{n+1}, B^n_j \rangle$ and to $\B{B}_j$ otherwise, breaking ties arbitrarily. This strategy achieves a reduction in computational complexity, but this reduction comes at a price: the following estimate for the best-of-two rule is essentially tight.

\begin{theorem}\label{twobinupbound}
Fix $d,k \in \N$ with $k \ge 2$, and let $\mu$ be any probability distribution on $\Bb^d\subset \R^d$. If we partition $\B{V}(\mu)$ into $k$ bins using the best-of-two rule, then almost surely,
\[ \limsup_{T\to\infty} \left(\frac{D(T)}{\log T} \right) =O(1).\]
\end{theorem}
Note that the best-of-two rule is identical to the inner product rule when $k = 2$. That Theorem~\ref{twobinupbound} is essentially best-possible when $k \ge 3$ is evidenced by the simple observation that if $\mu$ is the uniform distribution on $\Bb^1=[-1,1]$, then with high probability, there exists an interval of length $\Omega(\log T)$ in the first $T$ steps where we repeatedly choose the same pair of bins and only see numbers exceeding $1/2$. To prove Theorem~\ref{twobinupbound}, we shall show that the best-of-two rule enforces `self-correction'. Similar methods based on self-correction have recently been used to answer some long-standing questions about random graph processes; see~\citep{bohman, morris, keevash}, for example.
 
This paper is organised as follows. We give the proof of Proposition~\ref{lowbound} in Section~\ref{s:lowbound}. Section~\ref{s:ip} is devoted to analysing the inner product rule. We then address the best-of-two rule in Section~\ref{s:2b}. We finally conclude this note with a discussion of some open problems in Section~\ref{s:conc}. For the sake of clarity of presentation, we systematically omit floor and ceiling signs whenever they are not crucial.

\section{Lower bounds}\label{s:lowbound}
This section is devoted to the proof of Proposition~\ref{lowbound}, our strategy-agnostic lower bound. For completeness, we first record the following fact about the size of `slices' of the $d$-dimensional unit ball.
\begin{proposition}\label{slice}
For any $d\in\N$, there exist constants $C, c >0$ such that for any $e \in \Ss^{d-1}$ and any $0 \le b \le 1$,
\[c b \le \lambda_d (\{x: x \in \Bb^d \text{ and } \lvert\langle x,e \rangle\rvert \le b\}) \le C b. \eqno\qed\]
\end{proposition}

We now prove Proposition~\ref{lowbound}.

\begin{proof}[Proof of Proposition~\ref{lowbound}]
First, suppose that $k=2$ and consider any partitioning strategy that partitions $\B{V}(\mu)$ into two bins. Let $\delta_n = B_1^n - B_2^n$ and write $e_n$ for the unit vector in the direction of $\delta_n$. Note that we have
\[\nm{\delta_{n+1}}^2 = \nm{\delta_{n}}^2 + \nm{V}^2 + 2\nm{\delta_{n}}\langle V, e_n \rangle,\]
where $V=V_{n+1}$ if the strategy assigns $V_{n+1}$ to $\B{B}_1$ and $V=-V_{n+1}$ otherwise.

Consider the event $E_{n+1} = \{ 1/2 \le \nm{\delta_{n+1}}^2 - \nm{\delta_{n}}^2 \le 5/4\}$. We claim that regardless of the partitioning strategy used, we have
\begin{equation}\label{En}
\P(E_{n+1} \Cond \delta_n) \ge \frac{c_d}{\nm{\delta_{n}}},
\end{equation} 
where $c_{d} >0$ is a constant depending on the dimension $d$ alone. Indeed, regardless of which bin we assign $V_{n+1}$ to, if $\langle V_{n+1}, e_n \rangle \in [-(8\nm{\delta_{n}})^{-1}, (8\nm{\delta_{n}})^{-1}]$ and $\nm{V_{n+1}} \ge 3/4$, then $E_{n+1}$ holds. Since $\mu$ is the uniform distribution and $\B{V}(\mu)$ is an i.i.d.\ sequence, the claimed bound~\eqref{En} follows easily from Proposition~\ref{slice}.

Now, break the set $[T]$ into $r = T/m$ disjoint blocks $\Scr{T}_1, \Scr{T}_2, \dots, \Scr{T}_{r}$ each of length $m$ for some $m = m(T)$ that grows slowly with $T$ (and will be specified later). We say that a block $\Scr{T}$ is \emph{good} if $\nm {\delta_n}^2 \ge m/2$ for some $n \in \Scr{T}$. Now, for $i \in[r]$, consider the block $\Scr{T}_i = \{t+1, t+2, \dots, t+m\}$ with $t = (i-1)m$ and note, writing $F_i$ for the event that $\nm{\delta_t}^2 < m/2$, that
\[
\P(\Scr{T}_i \text{ is good} \Cond F_i) \ge \P(E_{t+1} \cap E_{t+2} \cap \dots \cap E_{t+m} \Cond F_i).
\]
Using~\eqref{En} and the fact that $\B{V}(\mu)$ is an i.i.d.\ sequence, we see that
\[\P(E_{t+j+1} \Cond F_i \cap E_{t+1} \cap E_{t+2} \cap \dots \cap E_{t+j}) \ge \frac{c_d}{\sqrt{m/2 + 5j/4}} \ge \frac{c_d}{2\sqrt{m}}.\]
It follows that
\[\P(\Scr{T}_i \text{ is good} \Cond F_i) \ge \left(\frac{c_d}{2\sqrt{m}}\right)^{m}.\]
Using the Markov property, we now deduce that
\[\P\left(D(T) < (m/2)^{1/2}\right) \le \prod_{i = 1}^{r} (1-\P(\Scr{T}_i \text{ is good} \Cond F_i)) \le \exp\left(-\frac{T}{m} \left(\frac{c_d}{2\sqrt{m}}\right)^{m}\right).\]
Applying the above bound with $m = \log T/ \log \log T$, we conclude that
\[\P\left(D(T) < \left({\frac{\log T}{2 \log \log T}}\right)^{1/2}\right) \le T^{-2}\] 
for all sufficiently large $T$; the proposition, in the case where $k = 2$, now follows from the Borel--Cantelli lemma.

In contrast to the situation with the arguments for upper bounds (that follow in subsequent sections), we may easily obtain a lower bound in the case where the number of bins exceeds two from the argument above that deals with the case of exactly two bins. Indeed when $k >2$, we proceed by `merging' the bins $\B{B}_1, \B{B}_2, \dots , \B{B}_{k'}$ and the bins $\B{B}_{k'+1}, \B{B}_{k'+2}, \dots, \B{B}_k$ into two auxiliary bins $\B{A}_1$ and $\B{A}_2$ respectively, where $k' = \lfloor k/2 \rfloor$; in other words, we set $A^n_1 = \sum_{i =1}^{k'} B^n_i$ and  $A^n_2 = \sum_{i = 1}^k B^n_i - A^n_1$ for each $n \in \N$. If $k$ is even, then we finish the proof as follows. By the argument above, it is clear that regardless of the partitioning strategy used, there exists an $n \in [T]$ for which
\[\nm {A_1^n - A_2^n} \ge \left({\frac{\log T}{2 \log \log T}}\right)^{1/2}\]
with probability at least $1-T^{-2}$; the result now follows from the triangle inequality. If $k$ is odd on the other hand, then the result follows from an analogous argument where we track
$\nm {(1+1/k')A_1^n - A_2^n}$ instead of $\nm {A_1^n - A_2^n}$.
\end{proof}

\section{The inner product rule}\label{s:ip}
We shall analyse the inner product rule in this section. We need the following standard Chernoff-type bound; see~\citep{textbook} for a proof.
\begin{proposition}\label{chernoff}
If $X_{1},X_2,\dots,X_{n}$ are independent random variables taking values in $\{0,1\}$, then writing $X=\sum_{i=1}^{n}X_{i}$, we have
\[\P(X \le \E[X]/2) \le \exp\left(- \E[X] / 8\right).\eqno\qed \]
\end{proposition}

We start by proving Theorem~\ref{ipupbound}.

\begin{proof}[Proof of Theorem~\ref{ipupbound}]
Given $m \ge 0$, we wish to bound $\P(D(T) \ge m)$ from above. Somewhat surprisingly, this is harder to do in the case where $k \ge 3$ as opposed to when $k = 2$. Indeed, to control $D(T)$, we need to control the distance between each pair of bins; however, if we attempt to control these distances \emph{individually}, we quickly run into difficulties because we cannot say much about how the distance between a \emph{particular} pair of bins changes at each time (unless $k = 2$). The trick is to instead track the observable 
\[S_n = \sum_{1\le i<j \le k} \nm{\delta_n(i,j)}^2,\] 
where $\delta_n(i,j) = B_i^n - B_j^n$ for all $i,j \in [k]$.

First, writing $e_n(i,j)$ for the unit vector in the direction of $\delta_n(i,j)$, note that
\[\nm{\delta_{n+1}(i,j)}^2 = \nm{\delta_n(i,j)}^2 + \nm{V}^2 + 2\nm{\delta_{n}(i,j)}\langle V, e_n(i,j) \rangle,\] 
where 
\begin{enumerate}
\item $V = V_{n+1}$ if $V_{n+1}$ is assigned to $\B{B}_i$, 
\item $V = -V_{n+1}$ if $V_{n+1}$ is assigned to $\B{B}_j$, and 
\item $V = 0$ otherwise.
\end{enumerate}
In particular, under the inner product rule, we have
\[\nm{\delta_{n+1}(i,j)}^2 = \nm{\delta_n(i,j)}^2 + \nm{V_{n+1}}^2 - 2\nm{\delta_{n}(i,j)} \lvert \langle V_{n+1}, e_n(i,j)\rangle\rvert \]
if $V_{n+1}$ is assigned to either $\B{B}_i$ or $\B{B}_j$, and $\nm{\delta_{n+1}(i,j)} = \nm{\delta_n(i,j)}$ otherwise. Hence, if $V_{n+1}$ is assigned to some bin $\B{B}_h$, then
\begin{equation}
S_{n+1} - S_n \leq (k-1) - 2\sum_{i \in [k]: i\ne h}\nm{\delta_n({h,i})} \lvert \langle V_{n+1}, e_n(h,i)\rangle\rvert. \label{obsable}
\end{equation}

Next, writing $\ell = km^2/2$, note that
\[\P(D(T) \ge m) \le \P\left( \left(\max_{1 \le n \le T} S_n \right) \ge \ell \right);\]
indeed, as a consequence of the triangle inequality, we have 
\[2\nm{\delta_n(h,i)}^2 + 2\nm{\delta_n(h,j)}^2 \ge \nm{\delta_n(i,j)}^2\] for any $h, i , j \in [k]$; summing this estimate over all $h \in [k]$, we deduce that
\[\frac{2S_n}{k} \ge \max_{ 1\le i,j \le k}  \,\nm{\delta_n(i,j)}^2.
\]
Consequently, writing $E_n (\ell) = \{S_1 < \ell\} \cap \{S_2 < \ell\} \cap \dots \cap \{S_{n-1} < \ell\} \cap \{S_n \ge \ell\}$, we have
\[
\P(D(T) \ge m) \le \sum_{n=1}^T\P(E_n(\ell)).
\]
Next, set $r = \ell /2k$. Note that for $n \leq r$, we have $S_n \le n(k-1) < \ell$, so it follows that $\P(E_n(\ell)) = 0$. For $n \ge r+1$, we define
\[F_n(\ell) = \left\{ \frac{\ell} {2} \le S_{n-r} < \ell\right\} \cap \left\{\frac{\ell} {2} \le S_{n-r+1} < \ell\right\} \cap \dots \cap \left\{\frac{\ell} {2} \le S_{n-1} < \ell\right\} \cap \{S_n \ge \ell\}.\] 
Under the inner product rule, we know (see~\eqref{obsable}) that $S_{n+1} - S_n \le k-1$ for all $n \in \N$, so it is clear if $S_n \ge \ell$, then $S_{n-t} \ge \ell/2$ for each $t \in \{0, 1, \dots, r\}$. Therefore, it is clear that $E_n(\ell) \subset F_n(\ell)$ for each $n \ge r+1$, so
\begin{equation}
\P(D(T) \ge m) \le \sum_{n=r+1}^T\P(F_n(\ell)).\label{bd_strat}
\end{equation}

We shall estimate $\P(F_n(\ell))$ by studying how our observable can change in a single step using~\eqref{obsable}, which in turn will allow us to bound $\P(D(T) \ge m)$ using~\eqref{bd_strat}. We need slightly different arguments depending on whether or not the underlying distribution $\mu$ is well-behaved. The key difference between the two cases is that the probability that the change in our observable in a single step is `bad' decays with $\ell$ in the case where $\mu$ is well-behaved (see Claim~\ref{step-2}), but is merely bounded away from $1$ in general (see Claim~\ref{step-1}).

\textbf{Case 1: Arbitrary distributions.}
We first establish~\eqref{gen_mu} for an arbitrary probability distribution $\mu$ on $\Bb^d$. We proceed by induction over the dimension. The result is trivial in the case where $d = 1$ as the inner product rule coincides with the trivial one-dimensional strategy described in Section~\ref{s:results}. Now, suppose that $d > 1$ and that we have established the required bound in dimension $d-1$.

The starting point of our argument is the following observation.
\begin{lemma}~\label{decomp}
For any probability distribution $\mu$ on $\Bb^d$, either there exists a hyperplane $\C{H}\subset\R^d$ passing through the origin such that $\mu (\C{H} \cap \Bb^d) = 1$, or there exist disjoint measurable sets $\C{A}_1, \C{A}_2, \dots, \C{A}_d \subset \Bb^d$ and constants $c, p >0$ such that
\begin{enumerate}
\item for every unit vector $e \in \Ss^{d-1}$, there exists $i \in [d]$ such that $\lvert\langle x, e\rangle \rvert \ge c$ for all $x \in \C{A}_i$, and
\item $\mu(\C{A}_i) \ge p$ for each $1 \le i \le d$.
\end{enumerate}
\end{lemma}
\begin{proof}
We say that a point $x \in \Bb^d$ is \emph{$\mu$-heavy} if $\mu(\C{U}) > 0$ for every open neighbourhood $\C{U}$ of $x$. If the set of $\mu$-heavy points is contained in some hyperplane $\C{H}$ passing through the origin, then it follows by compactness that for any $\eps > 0 $, $\mu (\C{H}_\eps) = 1$, where $\C{H}_\eps$ is the set of points at distance less than $\eps$ from $\C{H}$; as $\mu$ is a probability measure, it follows that $\mu(\C{H}) = 1$. Therefore, we may suppose that there exist $\mu$-heavy points $x_1, x_2, \dots, x_d$ such that no hyperplane passing through the origin contains all of these points; in other words, we may assume that for every $e \in \Ss^{d-1}$, there exists an $i \in [d]$ such that $\lvert\langle x_i, e \rangle\rvert > 0$. This implies the conclusion of the lemma, once again by compactness.
\end{proof}

We now apply Lemma~\ref{decomp} to $\mu$: if $\mu (\C{H} \cap \Bb^d) = 1$ for some hyperplane $\C{H}$ passing through the origin, then we are done by induction; we may therefore assume that there exist $\C{A}_1, \C{A}_2, \dots, \C{A}_d$ and $c,p>0$ as promised by Claim~\ref{decomp}.

To bound $\P(F_n(\ell))$ from above, we first estimate, for each $t \ge 0 $, the probability of the event 
\[I_t(\ell) = \left\{S_{t+1} \le S_t - {c\sqrt{\ell}}/{k} \right\}.\]

\begin{claim}\label{step-1}
For each $t \ge 0$ and all sufficiently large $\ell$, we have 
\[\P(I_t(\ell) \Cond \{S_{t} \ge \ell/2\}) \ge p.\]
\end{claim}
\begin{proof}
Relabelling the bins if necessary, suppose that the largest distance between a pair of bins at time $t$ is the distance between the bins $\B{B}_1$ and $\B{B}_2$. If $S_t \ge \ell/2$, then it must be the case that $\nm{\delta_t(1,2)}^2 \ge S_t / k^2 \ge \ell / 2k^2$. We know from Lemma~\ref{decomp} that with probability at least $p$, we have
\[ \lvert \langle V_{t+1}, \delta_t(1,2)\rangle \rvert \ge \frac{\lvert \langle V_{t+1}, e_t(1,2)\rangle \rvert \sqrt{\ell}}{\sqrt{2}k} \ge \frac{c}{k}\sqrt{\frac{\ell}{2}}.
\]
Consider the bins $\B{B}_\pm$ for which the inner products $\langle V_{t+1}, B^t_\pm \rangle$ are maximal and minimal. By definition, $V_{t+1}$ gets assigned to $\B{B}_-$ under the inner product rule. Now, since
\[ \left\langle V_{t+1}, B_+^t - B_-^t\right\rangle \ge \lvert \langle V_{t+1}, \delta_t(1,2)\rangle \rvert \ge \frac{c}{k}\sqrt{\frac{\ell}{2}}, \]
it follows from~\eqref{obsable} that if $S_{t} \ge \ell/2$, then with probability at least $p$, we have
\[
S_{t+1} \le S_t + (k-1) - \frac{2c}{k}\sqrt{\frac{\ell}{2}} \le S_t - \frac{c\sqrt{\ell}}{k},
\]
where last inequality holds provide $\ell$ is sufficiently large; this proves the claim.
\end{proof}

Consider any interval of $r$ steps in which our observable lies in the range $[\ell/2, \ell]$ and note that since our observable increases by at most $k-1$ at each step, there are at most $rk^2/c\sqrt{\ell}$ steps in this interval where our observable decreases by at least ${c\sqrt{\ell}}/{k}$. Consequently, if $F_n(\ell)$ holds, then there are at most $rk^2/c\sqrt{\ell} \le rp/2$ different values of $t \in \{n-r, n-r+1, \dots, n-1\}$ for which the event $I_t(\ell)$ holds, provided $\ell$ is sufficiently large. Using the Markov property, we deduce from Claim~\ref{step-1} and Proposition~\ref{chernoff} that for all $n \in \N$, we have
\[\P(F_n(\ell)) \le \exp \left({-rp/8}\right).\] 

We know that $\P(D(T) \ge m) \le \sum_{n=1}^T\P(F_n(\ell))$, so it is now clear that for all $m \ge 0$, we have
\begin{equation}
\P(D(T) \ge m) \le T \exp (-rp/8),\label{DM-1}
\end{equation}
where $\ell = km^2/2$, $r = \ell / 2k$ and  $p > 0$ is a constant depending on $d$ and $\mu$ alone. It follows from~\eqref{DM-1} that $D(T) = O((\log T)^{1/2})$ with probability at least $1-T^{-2}$; the required bound~\eqref{gen_mu} now follows from the Borel--Cantelli lemma.

\textbf{Case 2: H\"older continuous distributions.}
We now show how we may improve on~\eqref{gen_mu} for well-behaved distributions. It turns out that if $\mu$ is H\"older continuous, then it is possible to say a lot more about how our observable changes in a single step than in the general case. 

The starting point in this case is to bound, for each $t \ge 0 $, the probability of the event 
\[J_t(\ell) = \{S_{t+1} - S_{t} \ge -k\}.\]

\begin{claim}\label{step-2}
If $\mu$ is H\"older continuous, then for each $t \ge 0$,
\[ \P(J_t(\ell) \Cond  \{S_{t} \ge \ell/2\}) \le \frac{C}{\ell^c},
\]
where $C, c>0$ are constants depending on $d$, $k$ and $\mu$ alone.
\end{claim}
\begin{proof}
If $V_{t+1}$ is assigned to some bin $\B{B}_h$, then, by~\eqref{obsable}, we have
\[
S_{t+1} - S_t \leq (k-1) - 2\sum_{i \in [k]: i\ne h}\nm{\delta_t({h,i})} \lvert \langle V_{t+1}, e_t(h,i)\rangle\rvert.
\]
From the triangle inequality,
\[\sum_{i \in [k]: i\ne h} \nm{\delta_t({h,i})} \geq \max_{1 \leq i,j \leq k} \,\nm{\delta_t({i,j})} \geq \frac{\sqrt{S_t}}{k} .\]
Hence, if it so happens that 
\[\lvert \langle V_{t+1}, e_t(i,j)\rangle\rvert \geq \frac{k^2}{\sqrt{S_t}}\] 
for all $1 \leq i,j \leq k$, then
\[S_{t+1} - S_t \leq (k-1) - 2\left(\frac{\sqrt{S_t}}{k}\right) \left(\frac{k^2}{\sqrt{S_t}}\right) < -k.\]
Therefore, it follows that
\[ \P(J_t(\ell) \Cond  \{S_{t} \ge \ell/2\}) \le \sum_{1 \le i,j \le k} \P\left(\lvert \langle V_{t+1}, e_t(i,j)\rangle\rvert \le \frac{k^2}{\sqrt{\ell/2}}\right).\]
As $\mu$ is a H\"older continuous probability distribution and $\B{V}(\mu)$ is an i.i.d.\ sequence, the claim now follows from Proposition~\ref{slice}.
\end{proof}

As before, since $S_{n+1} - S_n \le k-1$ for all $n \in \N$, it is clear that in any interval of $r$ steps in which our observable lies in the range $[\ell/2, \ell]$, there must exist at least $r/2$ steps in this interval where our observable decreases by at most $k-1$. Consequently, if $F_n(\ell)$ holds, then there must exist at least $r/2$ different values of $t \in \{n-r, n-r+1, \dots, n-1\}$ for which the event $J_t(\ell)$ holds. Using the Markov property, we deduce from Claim~\ref{step-2} that if $\mu$ is H\"older continuous, then we have
\[\P(F_n(\ell)) \le \binom{r}{r/2} \left(\frac{C}{\ell^c}\right)^{r/2} \]
for all $n \in \N$. It follows that for all $m \ge 0$, we have
\begin{equation}
\P(D(T) \ge m) \le T \binom{r}{r/2} \left(\frac{C}{\ell^c}\right)^{r/2},\label{DM-2}
\end{equation}
where $\ell = km^2/2$, $r = \ell / 2k$ and  $C,c > 0$ are constants depending on $d$, $k$ and $\mu$ alone; a simple calculation using~\eqref{DM-2} shows that $D(T) = O((\log T /\log \log T)^{1/2})$ with probability at least $1-T^{-2}$; the required bound~\eqref{nice_mu} in the case where $\mu$ is H\"older continuous now follows from the Borel--Cantelli lemma.
\end{proof}

Note that under the inner product rule, the vector $V_{n+1}$ is only ever assigned to a bin $\B{B}_i$ if $B_i^n$ lies on the convex hull of the set $\{B_1^n, B_2^n, \dots, B_k^n\}$. Much is known about the convex hulls of random subsets of $\R^d$ (see~\citep{hull}, for example), and it seems possible to us that the H\"older condition in Theorem~\ref{ipupbound} could be relaxed by carefully tracking the convex hull of the bins. However, we cannot altogether do away with some sort of `well-behavedness' condition: the inner product rule does not match the lower bound in Proposition~\ref{lowbound} in general, as evidenced by Proposition~\ref{badexist} which we prove below.

\begin{proof}[Proof of Proposition~\ref{badexist}]
Given $\omega\colon  \R_{>0} \to \R_{>0}$ that is both increasing and unbounded, we first fix a fast-growing sequence of `length-scales'. More precisely, we fix a sequence $\B{L} =(L_s)_{s \ge 1}$ of positive reals such that for all $s \in \N$, we have $L_s \ge 2$ and 
\[\omega\left(\exp\left( s^2 L_s^2\right) - 1\right) \ge 10s.\] 
Writing $T_s = \lfloor \exp( s^2 L_s^2) \rfloor$, this construction ensures that we have 
\[L_s \ge \frac{10(\log T_s)^{1/2}}{\omega(T_s)}\] 
for each $s \in \N$. Having constructed $\B{L}$, we define an atomic probability distribution $\mu_\B{L}$ on $\Bb^2$ with weight $(6/\pi^2) s^{-2}$ on the vector $(1/L_s, -1/2)$ for each $s \in \N$. 

We now define $\mu_\omega$ as follows. A vector drawn from $\mu_\omega$ is the vector $(0, 1/2)$ with probability $1/3$, uniformly distributed on $\Bb^2$ with probability $1/3$, and distributed according to $\mu_\B{L}$ with probability $1/3$.

We shall show, using an argument analogous to the one used to prove Proposition~\ref{lowbound}, that if we partition $\B{V}(\mu_\omega)$ into two bins $\B{B}_1$ and $\B{B}_2$ using the inner product rule, then
\begin{equation} \P\left( D(T_s) < \frac{ L_s} { 10} \right) \le s^{-2} \label{omegabound}
\end{equation}
for all sufficiently large $s \in \N$. It is then clear that almost surely,
\[\limsup_{T\to\infty}\left(\frac{D(T) \omega(T)}{(\log T)^{1/2}} \right) \ge 1.\]

We now prove that~\eqref{omegabound} holds for all sufficiently large $s \in \N$. To this end, fix $s \in \N$ and write $L = L_s$ and $T = T_s = \lfloor \exp( s^2 L^2) \rfloor$. Also, let $\delta_n = B_1^n - B_2^n$ for all $n \in \N$.

As before, we break the set $[T]$ into $r = T/L^2$ disjoint blocks $\Scr{T}_1, \Scr{T}_2, \dots, \Scr{T}_{r}$ each of length $L^2$; in other words, for $i \in[r]$, we have $\Scr{T}_i = \{t+1, t+2, \dots, t+L^2\}$, where $t = (i-1)L^2$. We say that a block $\Scr{T}$ is \emph{good} if $\nm {\delta_n} \ge L/10$ for some $n \in \Scr{T}$. For $i \in [r]$, writing $t = (i-1)L^2$, we denote by $F_i$ the event that $\nm{\delta_t} < L$. We deduce~\eqref{omegabound} from the following claim.
\begin{claim}\label{blockclaim}
For each $i \in[r]$, we have
\[
\P(\Scr{T}_i \text{ is good}\,\Cond F_i) \ge s^{-L^2}.
\]
\end{claim}
\begin{proof}
We bound the probability of a block being good by showing that in the span of a block, there is a reasonably good chance of walking, using an alternating sequence of the vectors $(0,1/2)$ and $(1/L, -1/2)$, a distance of about $L/10$ to the right starting from somewhere close to the origin. We make this precise below.

Writing $t = (i-1)L^2$, first consider the event $E_1$ that there exists a time $P \in \{ t+1 , t+2, \dots, t+10L \}$ at which we have $\nm{\delta_{P}} \le 10$. We claim that \[\P( E_1 \Cond F_i) \ge (100)^{-10L};\] 
indeed, this follows from the fact that for all $n \in \N$, we crudely have
\[\P(\{\nm{\delta_{n+1}} \le \nm{\delta_{n}} - 1/10|\} \Cond \{ \nm{\delta_n} \ge 10 \}) \ge 1/100\]
because $V_{n+1}$ is sampled from the uniform distribution on $\Bb^2$ with probability $1/3$. 

Next, let
\[ \C{S} = \{(x,y): 0 \le x \le 1 \text{ and} -1/4 \le y \le 0 \} \subset \R^2\]
and consider the event $E_2$ that there exists a time $Q \in \{P, P+1, \dots, P+10^5\}$ at which we have $\delta_{Q} \in \C{S}$. It is not hard to check that again, we crudely have
\[ \P( E_2 \Cond E_1 \cap F_i ) \ge 100^{-100}.\]
To see this, note that if $P$ exists, then it is possible to walk, while `respecting the inner product rule' throughout, from $\delta_P$ to the set $\C{S}$ using vectors of norm $1/2$ in at most $100$ steps; the claimed bound then follows by `enlarging' such a walk and using the uniform component of $\mu_\omega$.

Finally, consider the event $E_3$ that the vectors $V_{Q+1}, V_{Q+2}, \dots, V_{Q + L^2/5}$ are alternately the vectors $(0,1/2)$ and $(1/L, -1/2)$. It is easy to see from the definition of $\mu_\omega$ that
\[ \P(E_3 \Cond E_2 \cap E_1 \cap F_i) = \left(\frac{2}{3\pi^2s^2}\right)^{\frac{L^2}{10}}.\]

Since $\delta_Q \in \C{S}$ under $E_2 \cap E_1 \cap F_i$, if $E_3$ also holds, then a simple calculation shows that we have $\delta_{t+1} = \delta_t + V_{t+1}$ for each $t \in \{Q, Q+1, \dots, Q + L^2/5 - 1\}$ under the inner product rule; consequently, under $E_3 \cap E_2 \cap E_1 \cap F_i$, we have $\delta_{Q + L^2/5} = \delta_Q + (L/10,0)$. It is now clear, provided $s\in \N$ is sufficiently large, that we have
\[ \P(\Scr{T}_i \text{ is good}\,\Cond F_i) \ge \P(E_3 \Cond E_2 \cap E_1 \cap F_i) \P(E_2 \Cond E_1 \cap F_i) \P( E_1 \Cond F_i) \ge s^{-L^2}
\]
 with room to spare.
\end{proof}

Using the Markov property, we now deduce from Claim~\ref{blockclaim} that
\[
\P\left(D(T) < L/10 \right) \le \prod_{i = 1}^{r} (1-\P(\Scr{T}_i \text{ is good} \Cond F_i)) \le \exp\left(-\frac{T s^{-L^2}}{L^2} \right) \le s^{-2},
\]
where the last inequality holds provided $s$ is sufficiently large as $T = \lfloor \exp( s^2 L^2) \rfloor$. It is now clear that~\eqref{omegabound} holds for all sufficiently large $s\in \N$; the proposition follows.
\end{proof}

\section{The best-of-two rule}\label{s:2b}
In this section, we prove Theorem~\ref{twobinupbound}. Before turning to the proof, let us recall the following classical concentration inequality due to Azuma and Hoeffding.
\begin{proposition}
Let $(X_t)_{t \ge 0}$ be a supermartingale such that $| X_t - X_{t-1} | \le C$ for all $t \ge 1$. For all positive integers $N$ and all $m\ge0$, we have
\[ \P(X_N - X_0 \ge m) \le \exp \left (\frac{-m^2}{2NC^2} \right). \eqno\qed
\]
\end{proposition}

Armed with the Azuma--Hoeffding inequality, we are ready to prove Theorem~\ref{twobinupbound}.

\begin{proof}[Proof of Theorem~\ref{twobinupbound}]
To prove the result, we shall first show that the distance between any pair of bins is `self-correcting' under the best-of-two rule, and then use martingale techniques to track these distances.

We proceed by induction over the dimension. Consider the function $f_\mu\colon \Ss^{d-1} \to [0,1]$ defined by 
\[f_\mu(e) = \int_{\Bb^d} \left| \langle x,e \rangle \right| \,d \mu
\]
and define
\[ C_\mu = \inf_{e \in \Ss^{d-1}} f(e).\]
We claim that it suffices to prove the result in the case where $C_\mu > 0$. Indeed, if $C_\mu = 0$, then since $f$ is continuous and $\Ss^{d-1}$ is compact, we have $f_\mu(e) = 0$ for some $e \in \Ss^{d-1}$. In other words, if $C_\mu = 0$, then there exists a hyperplane $\C{H} \subset \R^d$ passing through the origin such that $\mu(\C{H} \cap \Bb^d) = 1$. When $d = 1$, this is equivalent to saying that $\mu(\{0\}) = 1$, in which case the result is trivial. When $d>1$, by identifying $\C{H} \cap \Bb^d$ with $\Bb^{d-1}$, it is clear that $\mu$ may be identified with a probability distribution supported on the $(d-1)$-dimensional unit ball, in which case we may proceed inductively.

Now, assume that $C_\mu > 0$. We shall show that with probability at least $1-T^{-2}$, we have $\nm{B_1^n - B_2^n} = O(\log T)$ for all $n \in [T]$; the result then follows from a union bound over all pairs of bins and the Borel--Cantelli lemma.

Let $\delta_n = B_1^n - B_2^n$ and let $A_n = \nm{\delta_n}^2$. Our first task will be to estimate the conditional expectation $\E[A_{n+1} - A_n \Cond A_n]$; we do this as follows. 

Recall  that given $V_{n+1}$, we choose two bins $\B{B}_i$ and $\B{B}_j$ uniformly at random from the set of all bins (without replacement) and assign $V_{n+1}$ to $\B{B}_i$, say, if $\langle V_{n+1}, B^n_i \rangle \le \langle V_{n+1}, B^n_j \rangle$. Let $E^+$ denote the event that the two bins chosen at time $n+1$ are precisely $\B{B}_1$ and $\B{B}_2$ and let $E^-$ denote the event that neither of these two bins is chosen at time $n+1$. Also, for $i \in \{3, \dots, k\}$, let $E_i$ denote the event that two bins chosen at time $n+1$ are $\B{B}_i$ and one of $\B{B}_1$ or $\B{B}_2$.

First, writing $e_n$ for the unit vector in the direction of $\delta_n$, we have
\[
\E[A_{n+1} - A_n \Cond \delta_n, E^+]  = \E\left[\nm{V_{n+1}}^2\right] - 2f_\mu(e_n)\sqrt{A_n}
\]
from which it follows that
\begin{equation}
\E[A_{n+1} - A_n \Cond A_n, E^+]\le 1-2C_\mu\sqrt{A_n}.\label{condchange}
\end{equation}
Next, as the bins $\B{B}_1$ and $\B{B}_2$ are untouched at time $n+1$ under $E^-$, we also have $\E[A_{n+1} - A_n \Cond A_n, E^-] = 0$. Finally, we observe the following.
\begin{claim}
For each $3 \le i \le k$, we have
\[\E[A_{n+1} - A_n \Cond A_n, E_i] \le 1.\]
\end{claim}
\begin{proof}
To simplify notation, let $V = V_{n+1}$, $U_1 = B^n_1$, $U_2 = B^n_2$ and $U_i = B^n_i$. To prove the claim, we decompose $E_i$ into the events
\begin{enumerate}
\item $E_i(-,-) = E_i \cap \{\langle V, U_i - U_1\rangle \le 0\} \cap \{\langle V, U_i - U_2\rangle \le 0\}$,
\item $E_i(-,+) = E_i \cap \{\langle V, U_i - U_1\rangle \le 0\} \cap \{\langle V, U_i - U_2\rangle > 0\}$,
\item $E_i(+,-) = E_i \cap \{\langle V, U_i - U_1\rangle > 0\} \cap \{\langle V, U_i - U_2\rangle \le 0\}$, and
\item $E_i(+,+) = E_i \cap \{\langle V, U_i - U_1\rangle > 0\} \cap \{\langle V, U_i - U_2\rangle > 0\}$.
\end{enumerate}

First, as the vector $V$ is deterministically assigned to the bin $\B{B}_i$ under $E_i(-,-)$, \[\E[A_{n+1} - A_n \Cond A_n, E_i(-,-)] = 0.\] Next, we claim that
\[\E[A_{n+1} - A_n \Cond A_n, E_i(-,+)] \le 1\]
and that
\[\E[A_{n+1} - A_n \Cond A_n, E_i(+,-)] \le 1.\]
Indeed, under $E_i(-,+)$ for example, it is clear that the best-of-two rule always assigns $V$ to either $\B{B}_i$ or $\B{B}_2$ (but never to $\B{B}_1)$; since we also have $\langle V, U_1 - U_2\rangle > 0$ under $E_i(-,+)$, the claim follows. Finally, under $E_i(+,+)$, the best-of-two rule never assigns $V$ to $\B{B}_i$, and $V$ is equally like to be assigned to either $\B{B}_1$ or $\B{B}_2$ because each of these bins is equally likely to be the other bin selected in addition to $\B{B}_i$. Therefore,
\begin{multline*}
\E[A_{n+1} - A_n \Cond \delta_n, E_i(+,+)] = \E \left[\nm{V}^2\right] + (1/2)\E\left[ 2\sqrt{A_n} \langle V, e_n \rangle \Cond \delta_n, E_i(+,+)\right]\\ + (1/2)\E\left[ 2\sqrt{A_n} \langle -V, e_n \rangle \Cond \delta_n, E_i(+,+)\right],
\end{multline*}
and consequently,
\[ 
\E[A_{n+1} - A_n \Cond A_n, E_i(+,+)] = \E \left[\nm{V}^2\right] \le 1.
\]

Putting these facts together, it follows that
\[\E[A_{n+1} - A_n \Cond A_n, E_i] \le 1, \]
proving the claim.
\end{proof}

It is now clear that $\E[A_{n+1} - A_n \Cond A_n, (E^+)^c] \le 1$. As $\P(E^+) \ge 1/k^2$, we deduce from~\eqref{condchange} that
\begin{equation}
\E[A_{n+1} - A_n \Cond A_n] \le 1-C\sqrt{A_n},\label{selfcor}
\end{equation}
where $C = 2C_\mu / k^2 \le 1$ is a positive constant depending on $d$, $k$ and $\mu$ alone.

With the benefit of hindsight, let $m = (100 \log T / C)^2$ and denote by $F$ the event that $A_n > 2m$ for some $n \in [T]$. To bound $\P(F)$ from above, we define a collection of stopping times as follows. Let $\C{L}_0 = 0$ and for each $j \in \N$, let
\begin{enumerate}
\item $\C{U}_j = \inf\{n : n \ge \C{L}_{j-1} \text{ and } A_n \ge m\}$, and
\item $\C{L}_j = \inf\{n : n \ge \C{U}_j \text{ and }  A_n < m\}$.
\end{enumerate}

If $F$ holds, then it is clear that there exists a $j \in \N$ such that $A_n > 2m$ for some $n \in [\C{U}_j, \C{L}_j \wedge T]$. Let $F_j$ denote the event that there exists an $n \in [\C{U}_j, \C{L}_j]$ such that $A_n > 2m$ and note, by the union bound, that $\P(F) \le \sum_{j = 1}^T \P(F_j)$. Therefore, to complete the proof, it suffices to show that $\P(F_j) = o(T^{-3})$ for each $1 \le j \le T$.

For concreteness, we show that $\P(F_1) = o(T^{-3})$; the same argument may be used to bound $\P(F_j)$ for any $j \le T$. In what follows, all inequalities will hold provided $T$ (and hence $m$) is sufficiently large. Writing $\C{U} = \C{U}_1$ and $\C{L} = \C{L}_1$, we define another stopping time 
\[\C{N} = \C{L} \wedge 
\inf\{n : n \ge \C{U} \text{ and }  A_n > 2m\}.\] Clearly, $\P(F_1) = \P(\C{N} < \C{L})$. Now, set $\ell = C\sqrt{m}/{2}$ and consider, for $t \ge 0$, the process 
\[X_t = A_{t+\C{U}}+t\ell.\] 
First, note that for each $t \in [0, \C{N} -\C{U})$,
\[ X_{t+1} - X_t = \ell + A_{t+1+\C{U}} - A_{t+\C{U}}, \]
so we consequently have
\[ \lvert X_{t+1} - X_t \rvert \le \ell + 1 +2\sqrt{2m},\]
where the inequality above is immediate from the definition of $\C{N}$. Next, we also have
\[\E[X_{t+1} - X_t \Cond X_t] = \ell + \E[A_{t + 1 + \C{U}} - A_{t + \C{U}} \Cond A_{t + \C{U}}] \le  \ell + 1 - C\sqrt{m}\]
for each $t \in [0, \C{N}-\C{U})$, where the last inequality follows from~\eqref{selfcor} and the definitions of $\C{U}$ and $\C{L}$. It is now clear that $(X_t)_{t\ge 0}$ with $t \in [0, \C{N}-\C{U}]$ is a supermartingale with increments bounded by $4\sqrt{m}$. Therefore, for any $N \in [0, \C{L} - \C{U}]$, by the Azuma--Hoeffding inequality,  we have
\begin{align*}
\P(\C{N} = N + \C{U}) &\le \P(X_N- X_0 \ge 2m + N\ell - X_0)\\
&\le \P(X_N- X_0 \ge m/2 + N\ell)\\
&\le \exp\left(\frac{-(m/2 + N\ell)^2}{16Nm} \right) = o\left(T^{-4}\right),
\end{align*}
where the last inequality holds uniformly in $N$. By applying the union bound over the (at most $T$) possible values of $N$, we obtain that $\P(F_1) = o(T^{-3})$. This completes the proof of Theorem~\ref{twobinupbound}.
\end{proof}

\section{Conclusion}\label{s:conc}
First, it would be nice to know under what conditions~\eqref{nice_mu} holds in general. We have proved this estimate for probability distributions satisfying a H\"older condition. At the other end of the spectrum, the same bound also holds for probability distributions supported on a finite number of atoms; in fact, it can be shown in this case that under the inner product rule, we \emph{deterministically} have $D(T) = O(1)$. We know from Proposition~\ref{badexist} that the inner product rule does not match the lower bound in Proposition~\ref{lowbound} in general, however.

Next, it is worth mentioning that the construction in Proposition~\ref{badexist} was designed specifically to be `bad' for the inner product rule; in particular, this construction does not improve on the strategy-agnostic lower bound in Proposition~\ref{lowbound}. It is therefore an intriguing problem to decide the following: given a probability distribution on the unit ball, does there exist a (distribution-specific) partitioning strategy that matches the lower bound in Proposition~\ref{lowbound} to within a constant factor? Of course, one can also ask the following (perhaps more difficult) question: is there a universal strategy that matches the lower bound in Proposition~\ref{lowbound} for every probability distribution on the unit ball?

Finally, it would also be good to improve the implicit constants in our results and quantitatively understand the influence of the number of bins on the problems at hand; indeed, it is natural to expect that the freedom to use more bins should offer better control. A careful analysis of our proofs shows that for the uniform distribution, the lower bound in Proposition~\ref{lowbound} and the upper bound in Theorem~\ref{ipupbound} differ by a multiplicative factor of $k$, roughly; bridging this gap remains an interesting problem.

\section*{Acknowledgements}
We would like to thank Solom Heddaya for helpful discussions about the nature of resource requests in large scale computation and Ganesh Ananthanaryanan for empirical studies of various partitioning strategies on real-world data. Our thanks also to Rob Morris for interesting discussions on self-correction.

\bibliographystyle{amsplain}
\bibliography{balance_rv}


\begin{dajauthors}
\begin{authorinfo}[ja]

  Juhan Aru\\
  Departement Mathematik, ETH Z\"urich, R\"amistrasse 101, 8092, Z\"urich, Switzerland\\
  juhan.aru\imageat{}math\imagedot{}ethz\imagedot{}ch
\end{authorinfo}
\begin{authorinfo}[bn]
 Bhargav Narayanan\\
 Department of Pure Mathematics and Mathematical Statistics, University of Cambridge, Wilberforce Road, Cambridge CB3\thinspace0WB, UK\\
 b.p.narayanan\imageat{}dpmms\imagedot{}cam\imagedot{}ac\imagedot{}uk
 \end{authorinfo}
\begin{authorinfo}[as]
Alex Scott\\
Mathematical Institute, University of Oxford, Andrew Wiles Building, Radcliffe Observatory Quarter, Woodstock Road, Oxford OX2\thinspace6GG, UK\\
scott\imageat{}maths\imagedot{}ox\imagedot{}ac\imagedot{}uk
\end{authorinfo}
\begin{authorinfo}[rv]
Ramarathnam Venkatesan\\
Mircosoft Research, Redmond WA 98052, USA\\
venkie\imageat{}microsoft.com
\end{authorinfo}
\end{dajauthors}

\end{document}